\documentclass[reqno]{amsart}
\usepackage[foot]{amsaddr}
\usepackage{graphicx}
\usepackage{subcaption}  

\usepackage{hyperref,url}

\usepackage{amssymb}
\usepackage{amscd}
\usepackage{amsthm}
\usepackage{setspace}
\usepackage[dvipsnames]{xcolor}
\usepackage{enumerate}
\usepackage{graphicx}
\usepackage{mathtools}
\usepackage{tcolorbox} 
\usepackage{siunitx}
\usepackage{tikz-cd}
\usepackage{bm,blkarray}
\usepackage{float}
\numberwithin{equation}{section}
\usepackage{biblatex}

\usepackage{mathtools}
\usepackage[tableposition=top]{caption}
\usepackage{booktabs,dcolumn}

\renewcommand\H{\mathbb{H}}

\newcommand\Z{\mathbb{Z}}

\newcommand{\quash}[1]{}

\renewcommand{\H}{\mathcal H}

\title{Posets of Hyper $(b,t)$-ary Partitions are Distributive Lattices}
\author{Elsa Selin Frankel}
\address{Department of Mathematics, Wellesley College,  Wellesley, MA 02481 USA.}
\email{ef110@wellesley.edu}
\subjclass[2020]{05A17.}
\keywords{distributive lattice, integer partitions, refinement order}

\newtheorem{theorem}{Theorem}[section]
\newtheorem{definition}[theorem]{Definition}
\newtheorem{lemma}[theorem]{Lemma}

\newtheorem{proposition}[theorem]{Proposition}

\newtheorem{remark}[theorem]{Remark}

\begin{document}

\begin{abstract}
    Posets of integer partitions, ordered by refinement, were criticized for unruly structural behavior by both Birkhoff and Ziegler. However, recent work of Propp, McConville, and Sagan demonstrated that such posets of \textit{hyperbinary partitions} are distributive, with posets of join irreducible elements isomorphic to the well-studied class of \textit{fence posets}. We provide a generalization of this result for \textit{hyper ($b,t$)-ary partitions}, where all parts are powers of some positive integer $b$, and multiplicities are restricted to $t\ge b$. Then, we show that posets of hyper $(b,t)$-ary partitions are indeed also distributive lattices. 
    
\end{abstract}

\maketitle

\section{Introduction}

Posets of integer partitions for some fixed positive integer $n$, ordered by refinement, were briefly introduced by Birkhoff \cite{MR227053} to provide an integer partition analog to the well-studied \textit{set partition lattice} $\Pi_n$. Combinatorial and topological results on $\Pi_n$ are extensively explored in both general and restricted contexts. Work of Calderbank, Hanlon, Robinson \cite{MR850222} and Sagan, Sundaram \cite{sagan2026orderedsetpartitionposets} looks at partitions with restricted block size; \cite{MR850222} focuses on set partitions with block side congruent to $i \bmod{k}$ and \cite{sagan2026orderedsetpartitionposets} explores partitions with block sizes divisible by some fixed $d>2$. Other work of Sundaram \cite{Sun94b, MR1310588}, more generally describes applications of $\Pi_n$. Unlike their set partition lattice counterparts, both Birkhoff and Ziegler \cite{MR847552} note that posets of integer partitions, ordered by refinement, are quite ``badly behaved,'' as such posets fail to be lattices for $n>4$. Ziegler  \cite{MR847552} exemplified the unruly nature of these posets, disproving a conjecture of Bj\"orner  \cite{MR570784} that they are Cohen-Macaulay for $n\ge 9$, and thus further pushing away from connections to $\Pi_n$. Presently, few positive results about the general case of these posest of integer partitions exist.\\

However, certain restricted cases later proved fruitful. Work of McConville, Propp, and Sagan \cite{MR5036743} highlight nice structural properties of a restricted case: posets of \textit{hyperbinary integer partitions}, ordered by refinement. Hyperbinary partitions of $n$ have all parts as some power of two, with multiplicities restricted to at most two. McConville, Propp, and Sagan show that these posets are in fact distributive lattices, with posets of join irreducible elements isomorphic to the well-known class of \textit{fence posets}. In our paper, we investigate a question posed in \cite{MR5036743}, on additional posets of restricted  integer partitions that admit a lattice structure, under refinement ordering. We find that posets of hyper $(b,t)$-ary partitions, where all parts are powers of some positive integer $b$ with multiplicity restricted by $t\ge b$ are distributive lattices under refinement ordering.\\ 


The generalization to $(b,t)$-ary partitions comes up through study of partition functions $p(n)$ in arithmetic combinatorics, by \cite{MR389751, MR1579066, MR3450137}. Dirdal introduced them in \cite{MR389751} through study of restricted $m$-ary partition functions $s_{m,q}(n)$, which count partitions of $n$ into parts that are powers of $m$ with multiplicities restricted by $q\ge m$. Specific study of the hyperbinary partition function $p_{2,2}(n)$ continues to offer intriguing insights, more famously known in correspondence to \textit{Stern's diatomic sequence} \cite{MR1579066}. For a careful review of historical work on these restricted partition functions and additive results, we refer the reader to the thesis of Blair \cite{MR3450137}.\\

In Section \ref{sec2} we introduce necessary background and notation for our work. Then, in Section \ref{sec3} we present our main results. First, we show that posets of hyper $(b,t)$-ary partitions are lattices, leaning on a lemma that restricts the number of valid refinements from a certain part of a chosen partition. We then finish up the proof by utilizing a 1990 result of Björner, Eidelman, Ziegler to show certain finite graded posets are lattices, by looking at pairs of elements that cover their meet. Finally, we move to establish distributivity using a forbidden sublattice argument, as described in Section \ref{sec2}. 


\section{Background}\label{sec2}

We define a \textit{partition} of a positive integer $n$ as $\lambda=\{\lambda_1,\hdots,\lambda_k\}$, written as $\lambda_1\ge\cdots\ge\lambda_k$ by convention, such that $\sum\lambda_i=n$. The set of all possible integer partitions of some $n$ will be denoted as $\overline{\H(n)}$. Compositions of $n$ are defined similarly, but each ordering of parts $\lambda_i$ is treated as a unique partition of $n$. We shall denote the set of integer compositions as $\H(n)$, with notation $\lambda=(\lambda_1,\hdots,\lambda_n)$ for individual partitions. Compositions will not be discussed extensively in this paper, but are noted for contrasting definitions.\\

Lattices may be defined from a poset $P$ if every pair of elements $x,y\in P$ has both a least upper bound $x\lor y$ and greatest lower bound $x\land y$. These elements are called the join and meet of the pair $x,y$ respectively. Once we have obtained a lattice, we can further define classes of lattices with desirable structure. We say a lattice $L$ is \textit{modular} when given $x,y,z\in L$ with $x\le y$ we have
\begin{align*}
    x\lor (z\land y)=(x\lor z)\land y.
\end{align*}
A more restricted notion is that of a \textit{distributive lattice} $L$ as a lattice, where the distributivity laws hold for joins and meets. Suppose $x,y,z\in L$, then we must have
\begin{align*}
    x\lor(y\land z) &= (x\lor y)\land(x\lor z)\\
    x\land(y\lor z) &= (x\land y)\lor(x\land z).
\end{align*}
For any given lattice $L$, it suffices to prove only one law. Beyond direct methods, we have the following forbidden sublattice argument to show a lattice is distributive, due to Gr\"atzer \cite{MR509213}.

\begin{theorem}[Gr\"atzer]
    A lattice $L$ is a distributive lattice if and only if $L$ does not contain the diamond lattice $M_3$ or the pentagon lattice $N_5$ as sublattices.
\end{theorem}

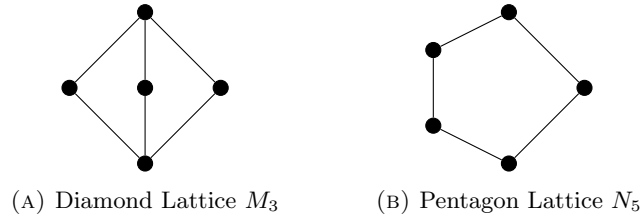
\begin{figure}[H]
    \begin{subfigure}{0.3\textwidth}
        \centering
        \begin{tikzpicture}
            \node[shape=circle,draw=black, fill=black, inner sep=2pt] (1) at (0,0) {};
            \node[shape=circle,draw=black, fill=black, inner sep=2pt] (2) at (0,1) {};
            \node[shape=circle,draw=black, fill=black, inner sep=2pt] (3) at (-1,1) {};
            \node[shape=circle,draw=black, fill=black, inner sep=2pt] (4) at (1,1) {};
            \node[shape=circle,draw=black, fill=black, inner sep=2pt] (5) at (0,2) {};
    
            \draw[-] (1) -- (2);
            \draw[-] (1) -- (3);
            \draw[-] (1) -- (4);
            \draw[-] (2) -- (5);
            \draw[-] (3) -- (5);
            \draw[-] (4) -- (5);
        \end{tikzpicture}
    \caption{Diamond Lattice $M_3$}
    \label{fig:subim1}
    \end{subfigure}
    \begin{subfigure}{0.3\textwidth}
        \centering
        \begin{tikzpicture}
            \node[shape=circle,draw=black, fill=black, inner sep=2pt] (1) at (0,0) {};
            \node[shape=circle,draw=black, fill=black, inner sep=2pt] (2) at (-1,0.5) {};
            \node[shape=circle,draw=black, fill=black, inner sep=2pt] (3) at (1,1) {};
            \node[shape=circle,draw=black, fill=black, inner sep=2pt] (4) at (-1,1.5) {};
            \node[shape=circle,draw=black, fill=black, inner sep=2pt] (5) at (0,2) {};
            
            \draw[-] (1) -- (2);
            \draw[-] (1) -- (3);
            \draw[-] (2) -- (4);
            \draw[-] (3) -- (5);
            \draw[-] (4) -- (5);
        \end{tikzpicture}
    \caption{Pentagon Lattice $N_5$}
    \label{fig:subim2}
    \end{subfigure}
    \caption{Forbidden Subposets for Distributive Lattices}
    \label{fig:figure2}
\end{figure}

Note that every distributive lattice is also modular, so a piece of the forbidden sublattice argument can be described for solely modular lattices.

\begin{theorem}[Gr\"atzer]
    A lattice $L$ is modular id and only if $L$ does not contain $N_5$ as a sublattice.
\end{theorem}

\section{Lattice of Hyper $(b,t)$-ary Partitions}\label{sec3}

\begin{definition}[Birkhoff]
    We say integer partitions $\lambda,\mu\in\H(n)$ are ordered by refinement $\mu\le\lambda$ if you can obtain all parts of $\mu$ from $\lambda$ by combining parts $\lambda_i$.
\end{definition}

Note that Birkhoff defines integer partition refinement ordering with a slight difference to ours -- his posets will be the dual of refinement posets on $\overline{\H(n)}$ in this paper. Orellana, Saliola, Schilling, and Zabrocki \cite{MR4895059} define a closely related ordering on integer partitions, which results in a lattice. However, recall that Birkhoff's posets will notably not be lattices, as we see in Figure \ref{fig:H(5)}. The main issue at play is the presence of multiple covers for the pair of partitions $\lambda=\{4,1\}$ and $\lambda'=\{3,2\}$, which primarily arises since the part $\lambda_1=4$ can be refined in two distinct ways, one of which coincides with the partition obtained from the unique refinement of part $\lambda'_1=3$. Therefore, if we should want to think of a class of restricted partitions that admit a lattice under refinement ordering, we ought to consider partitions where refinements of any given part are unique. One such class of restricted partitions will be the primary focus of our paper.

\begin{figure}[H]
    \centering
    \begin{tikzpicture}
        \node[shape=rectangle,draw=none, inner sep=3pt] (1) at (0,0) {$\{5\}$};
        
        \node[shape=rectangle,draw=none, inner sep=3pt] (2) at (-1,1) {$\{4,1\}$};
        \node[shape=rectangle,draw=none, inner sep=3pt] (3) at (1,1) {$\{3,2\}$};

        \node[shape=rectangle,draw=none, inner sep=3pt] (4) at (-1,2.25) {$\{3,1,1\}$};
        \node[shape=rectangle,draw=none, inner sep=3pt] (5) at (1,2.25) {$\{2,2,1\}$};

        \node[shape=rectangle,draw=none, inner sep=3pt] (6) at (0,3.25) {$\{2,1,1,1\}$};
        
        \node[shape=rectangle,draw=none, inner sep=3pt] (7) at (0,4.25) {$\{1,1,1,1,1\}$};
            
        \draw[-] (1) -- (2);
        \draw[-] (1) -- (3);

        \draw[-] (2) -- (4);
        \draw[-] (2) -- (5);
        \draw[-] (3) -- (4);
        \draw[-] (3) -- (5);

        \draw[-] (4) -- (6);
        \draw[-] (5) -- (6);

        \draw[-] (6) -- (7);
    \end{tikzpicture}
    \caption{Refinement Poset for $\overline{\H(5)}$}
    \label{fig:H(5)}
\end{figure}
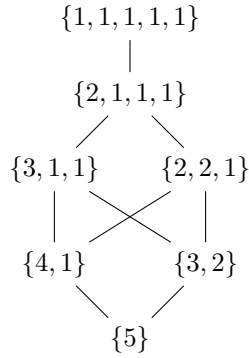

\begin{definition}[Hyper $(b,t)$-ary partitions] 
    Let $b,t\in\Z_{>0}$ with $t\ge b$. An integer partition $\lambda\in\H(n)$ is a hyper $(b,t)$-ary partition if all parts are powers of $b$, each with multiplicity at most $t$.
\end{definition}

We shall denote the set of all hyper $(b,t)$-ary partitions as $\H_{b,t}(n)$, for positive integers $t\ge b$. These partitions can be considered as a generalization to the well-studied \textit{hyperbinary partitions}, motivated by recent work of Propp, McConville, and Sagan \cite{MR5036743}. They describe how hyperbinary partitions admit a lattice structure under refinement ordering, and we shall show that the general hyper $(b,t)$-ary partitions also admit a lattice structure, beginning with the following lemma.

\begin{lemma}
    For a hyper $(b,t)$-ary partition $\lambda\in\H_{b,t}(n)$ with some part $\lambda_i$, either $\lambda_i$ can be refined to obtain a unique hyper $(b,t)$-ary partition $\mu\in\H_{b,t}$, or $\lambda_i$ cannot be refined to obtain a hyper $(b,t)$-ary partition at all.
\end{lemma}

\begin{proof}
    Let $\lambda\in\H_{b,t}(n)$ be a partition, and let $\mu$ be the partition obtained by refining some part $\lambda_i=b^i$ into $b$ parts of size $b^{i-1}$. If the multiplicity of parts with size $b^{i-1}$ in $\lambda$ is no more than $t-b$, then we would have that $\mu\in\H_{b,t}(n)$; otherwise there would be no such valid refinement at $\lambda_i$, as we exceed the multiplicity limit $t$. Now, we must show that there is no other refinement of $\lambda_i$ that could result in a hyper $(b,t)$-ary partition. Consider a positive integer $a\neq b^j$ for any $0\le j<i$, and notice that while a partition $\nu\in\H(n)$ can indeed be obtained from refining $\lambda_i$ into $a$ parts of size $b^i/a$, we find $\nu\notin\H_{b,t}(n)$ since $b^i/a$ is not a power of $b$.
\end{proof}

Notably $\H_{b,t}(n)$ is finite and bounded, with minimal element $\hat{0}\in\H_{b,t}(n)$ as the $b$-ary representation of $n$. The maximum element $\hat{1}\in\H_{b,t}(n)$ will be the longest partition of $n$ that can be expressed to respect the restrictions of both $b$ and $t$.

\begin{lemma}
    The longest hyper $(b,t)$-ary partition in $\H_{b,t}(n)$ is unique.
\end{lemma}

\begin{proof}
    Suppose we instead do have two distinct maximal partitions $\mu,\mu'\in\H_{b,t}(n)$, and let $\lambda\in\H_{b,t}(n)$. Suppose we can obtain $\mu$ and $\mu'$ from $\lambda$ by refining distinct sequences of parts from $\lambda$, where there is some part $\lambda_i$ which is refined in $\mu'$ but not in $\mu$, without loss of generality. Then there exists a refinement of $\lambda_i$ from $\mu$ to obtain a partition $\nu\in\H_{b,t}(n)$ such that $\mu<\nu$, since the sequence taken to obtain $\mu'$ implies that a refinement at $\lambda_i$ exists. This is impossible, as $\mu$ is defined as maximal. Therefore, we cannot have any distinct parts in either refinement sequence -- in other words, our sequences are identical. However, we still have $\mu\neq\mu'$, which can only occur if some part $\lambda_i$ in our sequence can be refined in two distinct ways. This is impossible by Lemma 3.3, and therefore we cannot have distinct maximal elements $\mu\neq\mu'$, as desired.
\end{proof}

\begin{lemma}
    $\mathcal{H}_{b,t}(n)$ under refinement ordering is graded and of finite rank.
\end{lemma}

\begin{proof}
    Define a rank function $\rho: \H_{b,t}(n)\to\Z_{\ge 0}$ so that $\rho(\lambda)$ is the number of refinements from individual parts between the minimal element $\hat{0}\in\H_{b,t}(n)$ and your chosen partition $\lambda\in\H_{b,t}(n)$. For instance, $\rho(\hat{0})=0$ and any refinement $\mu$ that can be obtained from a single refinement of a part in $\hat{0}$ has $\rho(\mu)=1$. Let $\lambda,\nu\in\H_{b,t}(n)$ be distinct partitions with $\lambda\le\nu$. Since $\nu$ is obtained from refining parts of $\lambda$, we find that $\rho(\lambda)\le\rho(\nu)$ as desired. Furthermore, if $\lambda\lessdot\nu$ then we simply take one more refinement from a single part of $\lambda$ to obtain $\nu$, so we additionally get that $\rho(\nu)=\rho(\lambda)+1$.\\

    To see why $\rho(\lambda)$ is well defined, let $\lambda\in\H_{b,t}(n)$ and define $C$ as a chain of length $k$ from $\hat{0}$ to $\lambda$. Suppose for a contradiction that we can construct another distinct chain $C'$ between $\hat{0}$ and $\lambda$ with length $m<k$. Now, recall from earlier that each refinement of some part $\hat{0}_i$ gives us $b$ new parts of size $\hat{0}_{i-1}$ and loses one part of size $\hat{0}_i$. Hence, the length of $\lambda$ can be written as $\ell(\lambda)=(b-1)k+\ell(\hat{0})$. However, this also implies that any sequence of $m$ refinements would give us a partition $\mu$ of length $\ell(\mu)=(b-1)m+\ell(\hat{0})<\ell(\lambda)$, so obtaining some $\mu=\lambda$ is impossible, and therefore we cannot construct our desired chain $C'$.
\end{proof}

\begin{theorem}[Björner, Eidelman, Ziegler, 1990]
    Let $P$ be a bounded poset of finite rank, where for any $x,y\in P$, if both $x$ and $y$ cover an element $z\in P$, then the join $x \lor y$ exists. Then $P$ is a lattice. 
\end{theorem}

\begin{theorem}
    $\H_{b,t}(n)$ under refinement ordering is a lattice.
\end{theorem}

\begin{proof}
    By Lemma 3.4, we can conclude that $\H_{b,t}(n)$ is bounded of finite rank. Consider $\lambda\in\H_{b,t}(n)$ and suppose two distinct partitions $\mu,\mu'\in\H_{b,t}(n)$ both cover $\lambda$. By Lemma 3.3, we know that each cover $\mu$ and $\mu'$ must be obtained from a single refinement of distinct parts $\lambda_i$ and $\lambda_j$ with $\lambda_i\neq\lambda_j$ respectively. Therefore, we can define a fourth partition $\nu$ that is obtained from then refining both $\lambda_j$ and $\lambda_i$ as in partitions $\mu$ and $\mu'$. To ensure that $\nu\in\H_{b,t}(n)$, it suffices to check our multiplicity condition. Since $i\neq j$, we know without loss of generality that a refinement of $\lambda_i$ only overlaps with $\lambda_j$ if $i=j+1$. Since $\mu\in\H_{b,t}(n)$ was obtained from a refinement at $\lambda_i$, it follows that the multiplicity of $\lambda_j$ will be at most $t-b$. Therefore the multiplicity of $\lambda_j$ in $\nu$ would be no more than $t-1$ after refinement at both parts. Otherwise if $|i-j|>1$, the assumption that $\mu,\mu'\in\H_{b,t}(n)$ implies that multiplicities of $\lambda_{i-1}$ and $\lambda_{j-1}$ do not exceed $t$, so that $\nu\in\H_{b,t}(n)$. Furthermore, since our partitions $\mu$ and $\mu'$ only differ by single refinements at $\lambda_i$ and $\lambda_j$ respectively, any cover of the pair must be obtained from refining at each part. Lemma $3.3$ tells us that there is only one way to refine $\lambda_i$ and $\lambda_j$ respectively, and consequentially $\nu$ is unique. By Theorem $3.5$, it finally follows that $\H_{b,t}(n)$ is a lattice.
\end{proof}

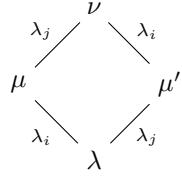
\begin{figure}[H]
    \[
    \begin{tikzpicture}{scale=0.75}
        \node[shape=rectangle,draw=none, inner sep=3pt] (1) at (0,0) {$\lambda$};
        \node[shape=rectangle,draw=none, inner sep=3pt] (2) at (-1,1) {$\mu$};
        \node[shape=rectangle,draw=none, inner sep=3pt] (3) at (1,1) {$\mu'$};
        \node[shape=rectangle,draw=none, inner sep=3pt] (4) at (0,2) {$\nu$};

        \node[shape=rectangle,draw=none, inner sep=3pt] (5) at (-0.7,0.3) {\scriptsize $\lambda_i$};
        \node[shape=rectangle,draw=none, inner sep=3pt] (6) at (-0.7,1.7) {\scriptsize $\lambda_j$};
        \node[shape=rectangle,draw=none, inner sep=3pt] (7) at (0.7,0.3) {\scriptsize $\lambda_j$};
        \node[shape=rectangle,draw=none, inner sep=3pt] (8) at (0.7,1.7) {\scriptsize $\lambda_i$};
            
        \draw[-] (1) -- (2);
        \draw[-] (1) -- (3);
        \draw[-] (3) -- (4);
        \draw[-] (2) -- (4);
    \end{tikzpicture}
    \]
    \caption{Hasse diagram of the subposet generated by $\lambda,\mu,\mu',\nu\in\H_{b,t}(n)$, with edges labeled by individual parts refined.}
    \label{fig:H(5)}
\end{figure}

To finally show our main result of distributivity, we shall first apply the following two results of Stanley \cite{MR2868112} to show that $\H_{b,t}(n)$ is modular. Then, having avoided an $N_5$ sublattice by Theorem 2.2 we will then show that $\H_{b,t}(n)$ has no $M_3$ sublattice, finally applying Theorem 2.1.

\begin{proposition}[Stanley]
    For a finite lattice $L$, the following statements are equivalent for $x,y\in L$:
    \begin{enumerate}
        \item $L$ is a graded lattice, with rank function $\rho$ satisfying $\rho(x)+\rho(y)\ge\rho(x\land y)+\rho(x\lor y)$.
        \
        \item if $x$ and $y$ both cover their meet $x\land y$, then they are also covered by their join $x\lor y$.
    \end{enumerate}
\end{proposition}

\begin{proposition}[Stanley]
    A finite lattice $L$, with rank function $\rho$, is modular if and only if for any elements $x,y\in L$ we have 
    \[
        \rho(x)+\rho(y)=\rho(x\land y)+\rho(x\lor y).
    \]
\end{proposition}

\begin{theorem}
    $\H_{b,t}(n)$ under refinement ordering is a distributive lattice.
\end{theorem}

\begin{proof}
    We shall proceed by means of Theorem 2.1. To show that our lattice avoids $N_5$, we must ensure that $\H_{b,t}(n)$ is modular. We shall first show that every pair of incomparable partitions in $\H_{b,t}(n)$ that cover their meet are also covered by their join. Consider $\mu,\mu'\in\H_{b,t}(n)$ and suppose both partitions cover their meet $\lambda\in\H_{b,t}(n)$. Then, say that $\mu$ and $\mu'$ are obtained by one distinct refinement of a part $\lambda_i$ and $\lambda_j$ respectively, where $\lambda_i\neq\lambda_j$. Therefore, we find that the join $\nu$ of $\mu$ and $\mu'$ can be obtained by refining $\lambda_j$ from $\mu$, without loss of generality, so that $\nu$ covers $\mu$; a similar argument proceeds to show $\nu$ covers $\mu'$. By Proposition 3.8, we consequentially find that $\rho(\mu)+\rho(\mu')\ge\rho(\lambda)+\rho(\nu)$.\\
    
    We shall again use the rank function $\rho$, where $\rho(\lambda)$ is the number of refinements from individual parts between the minimal element $\hat{0}\in\H_{b,t}(n)$ and your chosen partition $\lambda\in\H_{b,t}(n)$. Now, suppose instead that $\mu$ and $\mu'$ are covered by their join $\nu$. Then, say that $\mu$ and $\mu'$ differ by exactly one refinement from some part $\mu_i$ in $\mu$ and some part $\mu_j'$ in $\mu'$, such that $\mu_i\neq \mu_j$, respectively to obtain $\nu$. Let $\mu_i=b^k$ for some $k>0$ and notice that $\mu_i$ has been refined into $b$ parts of size $b^{k-1}$ in $\nu$. Then, since $\mu_j'\neq\mu_i$, we  retain those $b$ parts of size $b^{k-1}$ in $\mu'$. We can subsequently combine them to obtain a part of size $b^k$ in $\lambda$, and a similar argument exists for $\mu_j'$. Therefore, we find that a part of size $\mu_i$ can be refined from $\lambda$ to obtain $\mu'$ and $\mu_j'$ can be refined from $\lambda$ to obtain $\mu$, so $\mu$ and $\mu'$ must cover their meet, which is $\lambda$. Consequentially, we get $\rho(\mu)+\rho(\mu')=\rho(\lambda)+\rho(\nu)$, so $\H_{b,t}(n)$ is modular by Proposition 3.9. Since have a modular lattice, we have successfully avoided $N_5$ by Theorem $2.2$.\\

    To handle $M_3$, let $\mu,\mu',\mu''\in\H_{b,t}(n)$ be mutually incomparable, and suppose all three partitions cover their meet $\lambda$. Then distinct parts $\lambda_i,\lambda_j,\lambda_k$ must be refined to obtain $\mu,\mu'$, and $\mu''$ respectively from $\lambda$. Then, let $\nu$ be the join of $\mu,\mu'$ and $\mu''$. Without loss of generality, notice that to obtain $\nu$ from $\mu$, we must refine the other two parts: $\lambda_j$ and $\lambda_k$. Therefore, we cannot have any join $\nu$ that covers all three partitions $\mu,\mu'$ and $\mu''$, avoiding $M_3$. Since both $N_5$ and $M_3$ are avoided, we can conclude that $\H_{b,t}(n)$ is distributive.
\end{proof}

\begin{figure}[H]
    \centering
    \[
    \begin{tikzpicture}[scale=1.25]
        \node[shape=rectangle,draw=none, inner sep=3pt] (1) at (0,0) {$\lambda$};
        \node[shape=rectangle,draw=none, inner sep=3pt] (2) at (-1,1) {$\mu$};
        \node[shape=rectangle,draw=none, inner sep=3pt] (3) at (0,1) {$\mu'$};
        \node[shape=rectangle,draw=none, inner sep=3pt] (4) at (1,1) {$\mu''$};
        \node[shape=rectangle,draw=none, inner sep=3pt] (5) at (-1,2) {$\mu\lor\mu'$};
        \node[shape=rectangle,draw=none, inner sep=3pt] (6) at (0,2) {$\mu\lor\mu''$};
        \node[shape=rectangle,draw=none, inner sep=3pt] (7) at (1,2) {$\mu'\lor\mu''$};
        \node[shape=rectangle,draw=none, inner sep=3pt] (8) at (0,3) {$\nu$};

        \node[shape=rectangle,draw=none, inner sep=3pt] (9) at (-0.7,0.3) {\scriptsize $\lambda_i$};
        \node[shape=rectangle,draw=none, inner sep=3pt] (10) at (-1.2,1.5) {\scriptsize $\lambda_j$};
        \node[shape=rectangle,draw=none, inner sep=3pt] (11) at (0.7,0.3) {\scriptsize $\lambda_k$};
        \node[shape=rectangle,draw=none, inner sep=3pt] (12) at (0.25,0.6) {\scriptsize $\lambda_j$};
        \node[shape=rectangle,draw=none, inner sep=3pt] (13) at (-0.25,1.5) {\scriptsize $\lambda_i$};
        \node[shape=rectangle,draw=none, inner sep=3pt] (14) at (0.25,1.5) {\scriptsize $\lambda_k$};
        \node[shape=rectangle,draw=none, inner sep=3pt] (15) at (1.2,1.5) {\scriptsize $\lambda_j$};
        \node[shape=rectangle,draw=none, inner sep=3pt] (16) at (-0.5,1.8) {\scriptsize $\lambda_k$};
        \node[shape=rectangle,draw=none, inner sep=3pt] (17) at (0.5,1.8) {\scriptsize $\lambda_i$};

        \draw[-] (1) -- (2);
        \draw[-] (1) -- (3);
        \draw[-] (1) -- (4);
        \draw[-] (2) -- (5);
        \draw[-] (2) -- (6);
        \draw[-] (3) -- (5);
        \draw[-] (3) -- (7);
        \draw[-] (4) -- (7);
        \draw[-] (4) -- (6);
        \draw[-] (5) -- (8);
        \draw[-] (6) -- (8);
        \draw[-] (7) -- (8);
    \end{tikzpicture}
    \]
    \caption{Near-$M_3$ subgraph of some general refinement lattice $\H_{b,t}(n)$.}
    \label{fig:H(5)}
\end{figure}
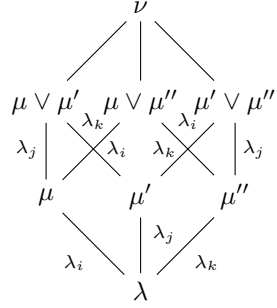

\begin{remark}
    As we can see from Figure \ref{fig:H(5)}, if $\lambda$ is a partition with three distinct covers $\mu,\mu',$ and $\mu''$ with join $\nu=\mu\lor\mu'\lor\mu''$, then the subposet given by the interval $[\lambda,\nu]$ will be isomorphic to the boolean lattice $B_3$.
\end{remark}

\section*{Acknowledgments}

The author thanks Richard P. Stanley for suggesting this project and offering useful advice throughout the process, along with Clara Chan for offering helpful feedback and review of my drafts. Additionally, she thanks Alex Postnikov, Ben Grant, Hugh Thomas and James Propp for helpful conversations throughout.

\nocite{*}

\printbibliography

@book{EC1,
  author = {Richard P. Stanley},
  year = {2011},
  title = {Enumerative Combinatorics: Volume 1},
  publisher = {Cambridge University Press}}

@article{knuth:1984,
  title={Literate Programming},
  author={Donald E. Knuth},
  journal={The Computer Journal},
  volume={27},
  number={2},
  pages={97--111},
  year={1984},
  publisher={Oxford University Press}
}

@book {MR227053,
    AUTHOR = {Birkhoff, Garrett},
     TITLE = {Lattice theory},
    SERIES = {American Mathematical Society Colloquium Publications},
    VOLUME = {Vol. XXV},
   EDITION = {Third},
 PUBLISHER = {American Mathematical Society, Providence, RI},
      YEAR = {1967}
}

@book {MR2868112,
    AUTHOR = {Stanley, Richard P.},
     TITLE = {Enumerative combinatorics. {V}olume 1},
    SERIES = {Cambridge Studies in Advanced Mathematics},
    VOLUME = {49},
   EDITION = {Second},
 PUBLISHER = {Cambridge University Press, Cambridge},
      YEAR = {2012},
     PAGES = {xiv+626}
}

@article {MR847552,
    AUTHOR = {Ziegler, G\"unter M.},
     TITLE = {On the poset of partitions of an integer},
   JOURNAL = {J. Combin. Theory Ser. A},
  FJOURNAL = {Journal of Combinatorial Theory. Series A},
    VOLUME = {42},
      YEAR = {1986},
    NUMBER = {2},
     PAGES = {215--222}
}

@article {MR570784,
    AUTHOR = {Bj\"orner, Anders},
     TITLE = {Shellable and {C}ohen-{M}acaulay partially ordered sets},
   JOURNAL = {Trans. Amer. Math. Soc.},
  FJOURNAL = {Transactions of the American Mathematical Society},
    VOLUME = {260},
      YEAR = {1980},
    NUMBER = {1},
     PAGES = {159--183}
}

@article {MR5036743,
    AUTHOR = {McConville, Thomas and Propp, James and Sagan, Bruce},
     TITLE = {Hyperbinary partitions and {$q$}-deformed rationals},
   JOURNAL = {Forum Math. Sigma},
  FJOURNAL = {Forum of Mathematics. Sigma},
    VOLUME = {14},
      YEAR = {2026},
     PAGES = {Paper No. e30, 25}
}

@article {MR389751,
    AUTHOR = {Dirdal, Gunnar},
     TITLE = {On restricted {$m$}-ary partitions},
   JOURNAL = {Math. Scand.},
  FJOURNAL = {Mathematica Scandinavica},
    VOLUME = {37},
      YEAR = {1975},
    NUMBER = {1},
     PAGES = {51--60}
}

@book {MR3450137,
    AUTHOR = {Blair, David Dakota},
     TITLE = {Counting {R}estricted {I}nteger {P}artitions},
      NOTE = {Thesis (Ph.D.)--City University of New York},
 PUBLISHER = {ProQuest LLC, Ann Arbor, MI},
      YEAR = {2015},
     PAGES = {75}
}

@article {MR1579066,
    AUTHOR = {Stern, M.},
     TITLE = {Ueber eine zahlentheoretische {F}unktion},
   JOURNAL = {J. Reine Angew. Math.},
  FJOURNAL = {Journal f\"ur die Reine und Angewandte Mathematik. [Crelle's
              Journal]},
    VOLUME = {55},
      YEAR = {1858},
     PAGES = {193--220}
}

@book {MR509213,
    AUTHOR = {Gr\"atzer, George},
     TITLE = {General lattice theory},
    SERIES = {Pure and Applied Mathematics},
    VOLUME = {75},
 PUBLISHER = {Academic Press, Inc. [Harcourt Brace Jovanovich, Publishers],
              New York-London},
      YEAR = {1978},
     PAGES = {xiii+381}
}

@misc{sagan2026orderedsetpartitionposets,
      title={Ordered set partition posets}, 
      author={Saga, B. E. n and Sundaram, S},
      year={2026},
      eprint={2506.23355},
      archivePrefix={arXiv},
      primaryClass={math.CO},
      url={https://arxiv.org/abs/2506.23355}, 
}

@article {MR850222,
    AUTHOR = {Calderbank, A. R. and Hanlon, P. and Robinson, R. W.},
     TITLE = {Partitions into even and odd block size and some unusual
              characters of the symmetric groups},
   JOURNAL = {Proc. London Math. Soc. (3)},
  FJOURNAL = {Proceedings of the London Mathematical Society. Third Series},
    VOLUME = {53},
      YEAR = {1986},
    NUMBER = {2},
     PAGES = {288--320}
}

@article {Sun94b,
    AUTHOR = {Sheila, S.},
     TITLE = {The Homology Representations of the Symmetric Group on Cohen-Macaulay Subposets of the Partition Lattice},
   JOURNAL = {Advances in Mathematics},
    VOLUME = {104},
      YEAR = {1994},
     PAGES = {225--296},
}

@incollection {MR1310588,
    AUTHOR = {Sundaram, Sheila},
     TITLE = {Applications of the {H}opf trace formula to computing homology
              representations},
 BOOKTITLE = {Jerusalem combinatorics '93},
    SERIES = {Contemp. Math.},
    VOLUME = {178},
     PAGES = {277--309},
 PUBLISHER = {Amer. Math. Soc., Providence, RI},
      YEAR = {1994},
}

@article {MR4895059,
    AUTHOR = {Orellana, Rosa and Saliola, Franco and Schilling, Anne and
              Zabrocki, Mike},
     TITLE = {The lattice of submonoids of the uniform block permutations
              containing the symmetric group},
   JOURNAL = {Semigroup Forum},
  FJOURNAL = {Semigroup Forum},
    VOLUME = {110},
      YEAR = {2025},
    NUMBER = {2},
     PAGES = {405--421}
}

\end{document}